\documentclass[a4paper,12pt,reqno]{amsart}
\usepackage[latin1]{inputenc}
\usepackage[english]{babel}
\usepackage{amsmath, amsthm, amssymb, amsopn, amsfonts, amstext, stmaryrd, enumerate, color, mathtools, hyperref, mathrsfs, enumitem, url}
\usepackage[margin=1in]{geometry}
\usepackage[final]{showkeys}

\hypersetup{
	colorlinks=true,
	linkcolor=blue,
	citecolor=red,
	urlcolor=black,
	linktoc=all
}

\newcommand{\CC}{\mathscr{C}}

\newcommand{\N}{\mathbb{N}}
\newcommand{\R}{\mathbb{R}}

\newcommand{\PV}{\mbox{\normalfont P.V.}}

\def\XXint#1#2#3{{\setbox0=\hbox{$#1{#2#3}{\int}$ }
\vcenter{\hbox{$#2#3$ }}\kern-.6\wd0}}

\theoremstyle{plain}
\newtheorem{definition}{Definition}[]
\newtheorem{theorem}[definition]{Theorem}
\newtheorem{proposition}[definition]{Proposition}

\theoremstyle{definition}
\newtheorem{remark}[definition]{Remark}

\renewcommand{\le}{\leqslant}

\renewcommand{\ge}{\geqslant}

\begin{document}

\title[On the growth of nonlocal catenoids]{On the growth of nonlocal catenoids}

\thanks{Both authors are members of INdAM/GNAMPA.
The first author was also a member of the Barcelona Graduate School of Mathematics. His research is supported by a Royal Society Newton International Fellowship and by
the MINECO grants MDM-2014-0445, MTM2014-52402-C3-1-P, and MTM2017-84214-C2-1-P.
The last author is supported by
the Australian Research Council Discovery Project 170104880 NEW ``Nonlocal
Equations at Work''. }

\author{Matteo Cozzi}
\author{Enrico Valdinoci}

\address[Matteo Cozzi]{
Department of Mathematical Sciences
\newline\indent University of Bath \newline\indent
Claverton Down, Bath BA2 7AY, 
United Kingdom}
\email{m.cozzi@bath.ac.uk}

\address[Enrico Valdinoci]{Department of Mathematics and Statistics
\newline\indent University of Western Australia \newline\indent
35 Stirling Highway, WA 6009 Crawley, Australia}
\email{enrico.valdinoci@uwa.edu.au}

\begin{abstract}
As well known, classical catenoids in~$\R^3$ possess logarithmic growth at infinity.
In this note we prove that the case of nonlocal minimal
surfaces is significantly different, and indeed all nonlocal catenoids
must grow at least linearly. More generally,
we prove that stationary sets for the nonlocal perimeter functional
which grow sublinearly at infinity are necessarily half-spaces.
\end{abstract}

\maketitle

\section{Introduction}

\noindent
The recent literature has taken into account
a number of important problems related to energies of
nonlocal type. One of the most challenging topics in this context
is given by the so-called ``nonlocal minimal surfaces'',
as introduced by~\cite{CRS10}. These objects are boundaries of sets
that minimize a nonlocal perimeter functional,
which takes into account all the interactions---weighted by
a rotation and translation invariant homogeneous kernel with
polynomial decay---between the points of a given set and the points
of its complement.

More precisely, for~$\alpha\in(0,1)$,
the kernel interaction between two disjoint measurable
sets~$E, F\subseteq \R^{n + 1}$
is given by
$$ I_\alpha(E,F):=\alpha(1-\alpha)\,\iint_{E\times F} \frac{dx\,dy}{|x-y|^{n + 1+\alpha}},
$$
and the $\alpha$-perimeter of~$E$ in a given domain~$\Omega$ is defined as
\begin{equation}\label{PER} {\rm Per}_\alpha(E,\Omega):=I_\alpha(E\cap\Omega,\Omega \setminus E)+
I_\alpha(E\cap\Omega, \R^{n + 1} \setminus (E \cup \Omega))+
I_\alpha(E\setminus\Omega,\Omega \setminus E),\end{equation}
which takes into account all the interactions of the set~$E$
with its complement, where at least one of the two interacting points
lies in the domain~$\Omega$.

The minimizers of~\eqref{PER} are
often called $\alpha$-minimal sets in~$\Omega$, and their boundaries
nonlocal, or fractional, $\alpha$-minimal surfaces in~$\Omega$.

If~$E$ is a minimizer of~\eqref{PER} in any given ball of~$\R^{n + 1}$,
we say that~$E$ is an~$\alpha$-minimal set in~$\R^{n + 1}$. 

We also consider here
a notion of $\alpha$-stationarity related to the fractional perimeter which is
very general, and basically relies on a definition ``in the viscosity
sense'' (in particular, the sets under consideration
are only assumed to be measurable, and no smoothness
or finite perimeter requirement is necessary in the framework that we adopt).
Namely, we say that~$E$ is
$\alpha$-stationary in a domain~$\Omega$ of~$\R^{n + 1}$
if for every~$x\in\partial E\cap\Omega$ the following conditions hold:
\begin{itemize}[leftmargin=*]
\item{\bf[Touching from the interior]}
for any set~$F$ with $C^2$ boundary in a neighborhood of~$x$
such that~$F \subseteq E$ and~$x \in \partial E \cap \partial F$,
we have that~$H_\alpha[F](x) \ge 0$;
\item{\bf[Touching from the exterior]}
for any set~$F$ with $C^2$ boundary in a neighborhood of~$x$
such that~$F \supseteq E$ and~$x \in \partial E \cap \partial F$,
we have that~$H_\alpha[F](x) \le 0$. 
\end{itemize}
In the setting above, the notation~$H_\alpha[F](x)$ stands for the nonlocal
mean curvature at a point~$x\in\partial F$, defined as
$$
H_\alpha[F](x) = \PV \int_{\R^{n + 1}} \frac{\chi_{\R^{n + 1} \setminus F}(y)
- \chi_F(y)}{|x - y|^{n + 1 + \alpha}} \,dy,$$
where, as usual, ``$\PV$'' denotes the Principal Value in the sense of Cauchy.

We also remark that if a set~$E$ is touched from the inside
at~$x\in\partial E$ by a~$C^2$ set~$F$, then the nonlocal mean
curvature of~$E$ is well defined and belongs to~$\R \cup \{ -\infty \}$.
Similarly, if a set~$E$ is touched from the outside
at~$x\in\partial E$ by a~$C^2$ set~$F$, then the nonlocal mean
curvature of~$E$ is well defined in~$\R \cup \{ +\infty \}$.
In particular, if~$E$ is $\alpha$-stationary its nonlocal mean curvature
is well defined and finite at points which can be touched from either inside or outside.

Notice that sets with~$C^2$ boundary in a neighborhood
of~$x$ for which~$H_\alpha[F](x)=0$ are~$\alpha$-stationary 
at~$x$ in the viscosity sense made precise above. 
Also, local minimizers
of the~$\alpha$-perimeter 
are~$\alpha$-stationary (see~\cite[Theorem~5.1]{CRS10}).
\medskip

Though the nonlocal perimeter recovers the classical perimeter,
as well as the nonlocal minimal surfaces recover the classical
minimal surfaces, up to normalizing constants
and in an appropriate limit sense as~$\alpha\nearrow1$, see~\cite{BBM01, D02, ADM11, CV13},
the nonlocal setting provides a number of 
extremely difficult problems
and several striking differences with respect to the classical case.
In particular, as~$\alpha\searrow0$, the functional in~\eqref{PER}
is related to a convex combination of Lebesgue measures,
with coefficients taking into account the behavior of the set at infinity,
see~\cite{MS02, DFPV13}. The lack of regularity properties
for this limit functional, combined with the predominant effects
of the energy contributions at infinity, produce serious
difficulties in the development of the regularity theory in dimension higher than two
and a series of new boundary stickiness effects, see~\cite{DSV17, BLV18}.
See also~\cite{CF17, DV18} for recent reviews on these and on related topics.
\medskip

One line of research related to nonlocal minimal surfaces consists
in detecting nonlocal counterparts of classical objects, such
as nonlocal catenoids,
see~\cite{DDPW14}, nonlocal helicoids, see~\cite{CDD16},
surfaces of constant nonlocal mean curvature, see~\cite{DDDV16, CFSW, CFW18a, CFW18b}.
\medskip

Interestingly, some of these nonlocal objects, such as the (double) helicoid, show exactly the same behavior of their classical counterparts, and others, such as the surfaces with constant nonlocal curvature, can be obtained by using delicate bifurcation methods from the corresponding classical surfaces. On the other hand, some other nonlocal surfaces, such as the nonlocal catenoids, exhibit important structural differences with respect to their classical analogues. As a matter of fact (see Figure~1 on page~115 of~\cite{DDPW14}) the nonlocal catenoid constructed in Theorem~1 of~\cite{DDPW14} can be seen as the boundary of an~$\alpha$-stationary set which possesses linear growth at infinity. More precisely,
such a two-dimensional surface can be described
for large~$r := \sqrt{ x^2_1 + x_2^2}$
as a two-leaves graph of the form~$|x_3| =\varphi(r)$, with
\begin{equation}\label{LINE}
\lim_{r\to+\infty}\frac{\varphi(r)}{r}=c\sqrt{1-\alpha},
\end{equation}
for some~$c>0$.

We stress that~\eqref{LINE} reveals an important structural
difference with respect to the local case, since the classical catenoid possesses logarithmic, rather than linear, growth at infinity. Therefore,
a natural question arising from these considerations
is whether or not one can construct nontrivial $\alpha$-stationary sets with
sublinear growth at infinity.\medskip

The main result of this paper is that the answer to this question is negative, namely a sublinear growth at infinity necessarily imposes that the $\alpha$-stationary set
is a half-space (differently from what happens in the local case).
More precisely, we prove the following result:

\begin{theorem} \label{mainthm}
Let~$n \ge 1$ be an integer and~$\alpha \in (0, 1)$. If~$E$ is an~$\alpha$-stationary set in~$\R^{n + 1}$ satisfying
\begin{equation} \label{bdEsublin}
\partial E \subseteq \Big\{ x = (x', x_{n + 1}) \in \R^{n + 1} : |x_{n + 1}| < \varphi(|x'|) \Big\},
\end{equation}
for some continuous function~$\varphi: [0, +\infty) \to (0, +\infty)$ such that
\begin{equation} \label{phisublin}
\lim_{r \rightarrow +\infty} \frac{\varphi(r)}{r} = 0,
\end{equation}
then~$E$ is a half-space.
\end{theorem}

To appreciate the optimality of the sublinear condition in Theorem~\ref{mainthm},
one can compare~\eqref{phisublin} with~\eqref{LINE}.

Notice that, when~$\partial E$ is also assumed to be a graph over the whole horizontal hyperplane, stronger results are available. In~\cite[Theorem~1.6]{CC17}, the flatness of~$\partial E$ is deduced allowing~$\varphi$ to have linear growth at infinity, with any slope. Subsequently, this result has been improved in~\cite[Theorem~1.3]{CFL18}, requiring only a one-sided bound on~$\partial E$---i.e., that only one of the inclusions of the forthcoming~\eqref{Egraph} is satisfied, with~$\varphi$ linear at infinity. See also~\cite{FV17,FarV17}, and again~\cite{CFL18} for related rigidity results of Bernstein and Moser type.

\section{Proof of Theorem~\ref{mainthm}}

\noindent
In light of assumption~\eqref{bdEsublin}, only two situations are possible, up to exchanging~$E$ with its complement: either~$\partial E$ is of \emph{catenoid-type}, meaning that
\begin{equation} \label{Ecatenoid}
E \subseteq \Big\{ x \in \R^{n + 1} : |x_{n + 1}| < \varphi(|x'|) \Big\},
\end{equation}
or~$\partial E$ is of \emph{graph-type}, i.e., it holds
\begin{equation} \label{Egraph}
\Big\{ x \in \R^{n + 1} : x_{n + 1} < - \varphi(|x'|) \Big\} \subseteq E \subseteq \Big\{ x \in \R^{n + 1} : x_{n + 1} < \varphi(|x'|) \Big\}.
\end{equation}
We deal with these two cases separately, each in one of the two subsequences that compose the remainder of the present section. More specifically, Theorem~\ref{mainthm} will be a consequence of the combination of the forthcoming Theorems~\ref{mainforcatenoids} and~\ref{mainforgraphs}.

\subsection{Catenoid-type~$\alpha$-stationary surfaces}

Here, we prove Theorem~\ref{mainthm} under the assumption that~\eqref{Ecatenoid} holds true. More precisely, we establish the following result.

\begin{theorem} \label{mainforcatenoids}
There exists no nontrivial~$\alpha$-stationary set~$E$ in~$\R^{n + 1}$ such that~\eqref{Ecatenoid} holds true for some continuous function~$\varphi: [0, +\infty) \to (0, +\infty)$ satisfying~\eqref{phisublin}.
\end{theorem}

Theorem~\ref{mainforcatenoids} is proved via a touching argument based on the maximum principle for~$\alpha$-stationary sets, used with the barrier constructed in the next proposition.

\begin{proposition} \label{barrierprop}
There exists a small~$\varepsilon_0 > 0$ such that, for every~$\varepsilon \in (0, \varepsilon_0]$, there exists a set~$F_\varepsilon$ of the form
\begin{equation} \label{Fepsdef}
F_\varepsilon := \Big\{ x \in \R^{n + 1} : |x_{n + 1}| < v_\varepsilon(x') \Big\},
\end{equation}
for some smooth, radially non-decreasing function~$v_\varepsilon: \R^n \to [\varepsilon, +\infty)$ satisfying
\begin{equation} \label{vepsbounds}
\begin{cases}
v_\varepsilon(x') = \varepsilon & \quad \mbox{for } x' \in B'_1 \\
\varepsilon \le v_\varepsilon(x') \le \varepsilon |x'| & \quad \mbox{for } x' \in B'_2 \setminus B'_1 \\
v_\varepsilon(x') = \varepsilon |x'| & \quad \mbox{for } x' \in \R^n \setminus B'_2,
\end{cases}
\end{equation}
and~$v_\varepsilon \le v_{\varepsilon'}$ in~$\R^n$ for every~$0 < \varepsilon \le \varepsilon' \le \varepsilon_0$, such that
\begin{equation} \label{HFeps>0}
H_\alpha[F_\varepsilon](x) > 0 \quad \mbox{for every } x \in \partial F_\varepsilon.
\end{equation}
\end{proposition}

\begin{remark}
Proposition~\ref{barrierprop} reveals a structural difference between
the classical and the nonlocal cases: for instance,
in dimension~$2$, the (mean) curvature of the set~$F_\varepsilon$
necessarily vanishes at points~$(x',x_2)\in\partial F_\varepsilon$
with~$x'\in\R\setminus (-2 ,2)$, and so the analogue of~\eqref{HFeps>0}
cannot hold in the classical case. In higher dimensions this difference
is even more contrasting, as the standard mean curvature of~$F_\varepsilon$ is actually negative at these points. For other barriers of purely nonlocal
character see also Proposition~7.3 in~\cite{DSV17},
Proposition~5 in~\cite{MR3778164},
and Proposition~3.1 and Lemma~8.1 in~\cite{CDNV}.
\end{remark}

\begin{proof}[Proof of Proposition~\ref{barrierprop}]
Let~$\eta \in C_c^\infty(\R^n)$ be a radially non-increasing cut-off function satisfying~$\eta = 1$ in~$B'_1$ and~$\eta = 0$ in~$\R^n \setminus B'_2$. We define
$$
v(x') := \eta(x') + (1 - \eta(x')) |x'|
$$
and, given~$\varepsilon > 0$,
$$
v_\varepsilon(x') := \varepsilon v(x'),
$$
for every~$x' \in \R^n$. It is immediate to check that~$v_\varepsilon$ is a smooth, radially non-decreasing function verifying~\eqref{vepsbounds} and
\begin{equation} \label{vepsreg}
\|\nabla v_\varepsilon\|_{L^\infty(\R^n)} + \| D^2 v_\varepsilon \|_{L^\infty(\R^n)} \le C \varepsilon, 
\end{equation}
for some dimensional constant~$C > 0$. We claim that the set~$F_\varepsilon$ defined by~\eqref{Fepsdef} satisfies~\eqref{HFeps>0}.

Let~$\CC_\varepsilon$ be the cone given by
$$
\CC_\varepsilon := \Big\{ x \in \R^{n + 1} : - \varepsilon |x'| < x_{n + 1} < \varepsilon |x'| \Big\}.
$$
In~\cite[Section~10]{DDPW14}, it is proved that
\begin{equation} \label{supercones}
H_\alpha[\CC_\varepsilon](x) = \frac{M(\varepsilon)}{|x|^\alpha} \quad \mbox{for every } x \in \partial \CC_\varepsilon \setminus \{ 0 \},
\end{equation}
for some continuous function~$M: (0, +\infty) \to \R$ such that
\begin{equation} \label{Mexplodes}
\lim_{\varepsilon \to 0^+} M(\varepsilon) = +\infty.
\end{equation}
In particular, there exists~$\bar{\varepsilon} \in (0, 1/2)$ such that~$M(\varepsilon) > 0$ for every~$\varepsilon \in (0, \bar{\varepsilon})$. In the following, we will always assume~$\varepsilon$ to lie in such interval.

Let~$x \in \partial F_\varepsilon \setminus B_4$. Since~$F_\varepsilon \setminus B_3 = \CC_\varepsilon \setminus B_3$, we have that~$x \in \partial \CC_\varepsilon$ and
\begin{align*}
\left| H_\alpha[F_\varepsilon](x) - H_\alpha[\CC_\varepsilon](x) \right| & = \left| \int_{B_3} \frac{\chi_{\R^{n + 1} \setminus F_\varepsilon}(y) - \chi_{F_\varepsilon}(y) - \chi_{\R^{n + 1} \setminus \CC_\varepsilon}(y) + \chi_{\CC_\varepsilon}(y)}{|x - y|^{n + 1 + \alpha}} \, dy \right| \\
& \le 2 \int_{B_3} \frac{dy}{|x - y|^{n + 1 + \alpha}} \le \frac{4^{n + 2 + \alpha} |B_3|}{|x|^{n + 1 + \alpha}} \le \frac{C}{|x|^\alpha}.
\end{align*}
From now on,~$C$ will indicate a constant larger than~$1$ that depends only on~$n$ and~$\alpha$, and whose value may possibly change from line to line. By the above estimate and~\eqref{supercones}, we deduce that
$$
H_\alpha[F_\varepsilon](x) \ge H_\alpha[\CC_\varepsilon](x) - \left| H_\alpha[F_\varepsilon](x) - H_\alpha[\CC_\varepsilon](x) \right| \ge \frac{M(\varepsilon) - C}{|x|^\alpha}.
$$
Hence, thanks to~\eqref{Mexplodes}, there exists~$\varepsilon_1 \in (0, 1)$ such that
\begin{equation} \label{claiminthelarge}
H_\alpha[F_\varepsilon](x) > 0 \quad \mbox{for every } x \in \partial F_\varepsilon \setminus B_4,
\end{equation}
provided~$\varepsilon \in (0, \varepsilon_1]$.

We now deal with the case~$x \in \partial F_\varepsilon \cap B_4$. Let~$x' \in B_4'$. Without loss of generality, we may assume that~$x_{n + 1} > 0$, i.e., that~$x_{n + 1} = v_\varepsilon(x')$. Through computations analogous to those made in~\cite[Lemma~4.1]{CFSW} for~$n = 1$, we write the~$\alpha$-mean curvature of~$F_\varepsilon$ at~$x$ as
\begin{align*}
H_\alpha[F_\varepsilon](x) & = 2 \, \PV \int_{\R^n} G \! \left( \frac{v_\varepsilon(x') - v_\varepsilon(y')}{|x' - y'|} \right) \frac{dy'}{|x' - y'|^{n + \alpha}} \\
& \quad + 2 \int_{\R^n} \left\{ G(+\infty) - G \! \left( \frac{v_\varepsilon(x') + v_\varepsilon(y')}{|x' - y'|} \right) \right\} \frac{dy'}{|x' - y'|^{n + \alpha}},
\end{align*}
with
$$
G(t) = \int_0^t \frac{d\tau}{(1 + \tau^2)^{\frac{n + 1 + \alpha}{2}}} \quad \mbox{for } t \in \R.
$$
Observe that~$G$ is an odd, increasing functions that satisfies~$G(0) = 0$ and
\begin{equation} \label{G'le1}
|G(t_1) - G(t_2)| \le |t_1 - t_2| \quad \mbox{for every } t_1, t_2 \in \R.
\end{equation}

To estimate the integrals above, we distinguish between the two cases~$y' \in B'_\varepsilon(x')$ and~$y' \in \R^n \setminus B'_\varepsilon(x')$. In the former case, thanks to~\eqref{G'le1} and~\eqref{vepsreg} we have
\begin{equation} \label{insidetech1}
\begin{aligned}
& \left| \PV \int_{B_\varepsilon(x')} G \! \left( \frac{v_\varepsilon(x') - v_\varepsilon(y')}{|x' - y'|} \right) \frac{dy'}{|x' - y'|^{n + \alpha}} \right| \\
& \hspace{30pt} = \left| \int_{B_\varepsilon(x')} \left\{ G \! \left( \frac{v_\varepsilon(x') - v_\varepsilon(y')}{|x' - y'|} \right) - G \! \left( \frac{\nabla v_\varepsilon(x') \cdot (x' - y')}{|x' - y'|} \right) \right\} \frac{dy'}{|x' - y'|^{n + \alpha}} \right| \\
& \hspace{30pt} \le  \int_{B_\varepsilon(x')} \frac{\left| v_\varepsilon(x') - v_\varepsilon(y') - \nabla v_\varepsilon(x') \cdot (x' - y')\right|}{|x' - y'|^{n + 1 + \alpha}} \, dy' \le C \varepsilon \int_{B_\varepsilon} \frac{dz'}{|z'|^{n - 1 + \alpha}} \\
& \hspace{30pt} \le C \varepsilon^{2 - \alpha}.
\end{aligned}
\end{equation}
Also, from~\eqref{vepsbounds} it easily follows that, for all~$y' \in B_\varepsilon(x')$,
\begin{align*}
G(+\infty) - G \! \left( \frac{v_\varepsilon(x') + v_\varepsilon(y')}{|x' - y'|} \right) & = \int_{\frac{v_\varepsilon(x') + v_\varepsilon(y')}{|x' - y'|}}^{+\infty} \frac{d\tau}{(1 + \tau^2)^{\frac{n + 1 + \alpha}{2}}} \\
& \ge 2^{- \frac{n + 1 + \alpha}{2}} \int_{\frac{10 \varepsilon}{|x' - y'|}}^{+\infty} \frac{d\tau}{\tau^{n + 1 + \alpha}} \ge \frac{C^{-1}}{\varepsilon^{n + \alpha}} |x' - y'|^{n + \alpha},
\end{align*}
and thus
\begin{equation} \label{insidetech2}
\int_{B'_\varepsilon(x')} \left\{ G(+\infty) - G \! \left( \frac{v_\varepsilon(x') + v_\varepsilon(y')}{|x' - y'|} \right) \right\} \frac{dy'}{|x' - y'|^{n + \alpha}} \ge \frac{C^{-1}}{\varepsilon^\alpha}.
\end{equation}
On the other hand, using again~\eqref{vepsreg},~\eqref{G'le1}, and the monotonicity of~$G$, we find that 
\begin{align*}
& \int_{\R^n \setminus B_\varepsilon(x')} \left\{ G \! \left( \frac{v_\varepsilon(x') - v_\varepsilon(y')}{|x' - y'|} \right) + G(+\infty) - G \! \left( \frac{v_\varepsilon(x') + v_\varepsilon(y')}{|x' - y'|} \right) \right\} \frac{dy'}{|x' - y'|^{n + \alpha}} \\
& \hspace{40pt} \ge \int_{\R^n \setminus B_\varepsilon(x')} \frac{- G \! \left( \| \nabla v_\varepsilon \|_{L^\infty(\R^n)} \right) + 0}{|x' - y'|^{n + \alpha}} \, dy' \ge - G \! \left( C \varepsilon \right) \int_{\R^n \setminus B_\varepsilon} \frac{dz'}{|z'|^{n + \alpha}} \ge - C \varepsilon^{1 - \alpha}.
\end{align*}

By putting together the last estimate with~\eqref{insidetech1} and~\eqref{insidetech2}, we conclude that
$$
H_\alpha[F_\varepsilon](x) \ge \frac{2}{\varepsilon^{\alpha}} \left( C^{-1} - C \varepsilon - C \varepsilon^2 \right) \quad \mbox{for every } x \in \partial F_\varepsilon \cap B_4.
$$
In particular, there exists~$\varepsilon_2 \in (0, 1)$ such that
$$
H_\alpha[F_\varepsilon](x)  > 0 \quad \mbox{for every } x \in \partial F_\varepsilon \cap B_4,
$$
provided~$\varepsilon \in (0, \varepsilon_2]$. The combination of this and~\eqref{claiminthelarge} gives~\eqref{HFeps>0}.
\end{proof}

With the result of Proposition~\ref{barrierprop} in hand, we are now ready to proceed with the proof of Theorem~\ref{mainforcatenoids}.

\begin{proof}[Proof of Theorem \ref{mainforcatenoids}]
We argue by contradiction and suppose that such an~$\alpha$-stationary set exists. Thanks to the properties of~$\varphi$ and hypothesis~\eqref{Ecatenoid}, for every~$\delta > 0$, there exists a constant~$C_\delta > 0$ such that
\begin{equation} \label{Esubeqlin}
E \subseteq \Big\{ x \in \R^{n + 1} : |x_{n + 1}| < C_\delta + \delta |x'| \Big\},
\end{equation}

Let now~$\varepsilon_0 > 0$ be the parameter found in Proposition~\ref{barrierprop} and choose~$\delta := \varepsilon_0 / 8$. We then consider the rescaled set~$E_\star := \lambda E$, with~$\lambda :=\varepsilon_0 / (8 C_{\varepsilon_0/8} )$. Of course,~$E_\star$ is~$\alpha$-stationary. Moreover, thanks to~\eqref{Ecatenoid} and~\eqref{Esubeqlin}, it satisfies
\begin{equation} \label{E*sublin}
E_\star \subseteq \left\{ y \in \R^{n + 1} : |y_{n + 1}| < \lambda \, \varphi \! \left( \frac{|y'|}{\lambda} \right) \right\} \cap \bigg\{ y \in \R^{n + 1} : |y_{n + 1}| < \frac{\varepsilon_0}{8} \left( 1 + |y'| \right) \bigg\}.
\end{equation}
Let~$F_\varepsilon$ and~$v_\varepsilon$ be as in Proposition~\ref{barrierprop}. Note that~$v_\varepsilon(x') \ge \frac{\varepsilon}{4} ( 1 + |x'| )$ for all~$x' \in \R^n$. Accordingly,~\eqref{E*sublin} gives that
\begin{equation} \label{1ststep}
E_\star \subseteq F_{\varepsilon_0/2}.
\end{equation}

Define now
$$
\varepsilon_\star := \inf \big\{ \varepsilon > 0 : E_\star \subseteq F_\varepsilon \big\}.
$$
We clearly have that~$\varepsilon_\star \le \varepsilon_0 / 2$, thanks to~\eqref{1ststep}. We claim that, in fact,
\begin{equation} \label{epsdagger=0}
\varepsilon_\star = 0.
\end{equation}
To check this, we argue by contradiction and suppose that~$\varepsilon_\star \in (0, \varepsilon_0 / 2]$. Of course,
\begin{equation} \label{EstarinFdagger}
E_\star \subseteq F_{\varepsilon} \quad \mbox{for all } \varepsilon \ge \varepsilon_\star
\end{equation}
and
\begin{equation} \label{EstarnotinFdagger}
E_\star \not\subseteq F_{\varepsilon} \quad \mbox{for all } \varepsilon \in (0,  \varepsilon_\star).
\end{equation}
In consequence of~\eqref{EstarnotinFdagger}, there exists a sequence of points~$\{x^{(j)}\}$ such that~$x^{(j)} \in E_\star \setminus F_{\varepsilon_\star - 1/j}$ for every large~$j \in \N$.

We claim that
\begin{equation} \label{xj'bounded}
\{ x^{(j)} \} \mbox{ is a bounded sequence}.
\end{equation}
If not, then by~\eqref{E*sublin},~\eqref{Fepsdef}, and~\eqref{vepsbounds} we would get that
$$
\lambda \, \varphi \! \left( \frac{|(x^{(j)})'|}{\lambda} \right) \ge |x^{(j)}_{n + 1}| \ge v_{\varepsilon_\star - \frac{1}{j}}((x^{(j)})') = \left( \varepsilon_\star - \frac{1}{j} \right) |(x^{(j)})'|,
$$
for infinitely many~$j$. Setting~$r_j := |(x^{(j)})'| / \lambda$, this yields that
$$
\limsup_{r \rightarrow +\infty} \frac{\varphi(r)}{r} \ge \limsup_{j \rightarrow +\infty}\frac{\varphi(r_j)}{r_j} \ge \varepsilon_\star > 0, 
$$
in contradiction with~\eqref{phisublin}. As a result,~\eqref{xj'bounded} holds true.

We infer from~\eqref{xj'bounded} that, up to a subsequence,~$\{ x^{(j)} \}$ converges to some~$x_\star \in \R^{n + 1}$. Necessarily,~$x_\star \in \partial F_{\varepsilon_\star} \cap \partial E_\star$. Since~$E_\star \subseteq F_{\varepsilon_\star}$ (by~\eqref{EstarinFdagger}),~$\partial F_{\varepsilon_\star}$ is smooth, and~$E_\star$ is~$\alpha$-stationary, we conclude that
$$
H_\alpha[F_{\varepsilon_\star}](x_\star) \le 0,
$$
in contradiction with~\eqref{HFeps>0}. Hence,~\eqref{epsdagger=0} must be true and, therefore,~$E_\star \subseteq F_\varepsilon$ for all~$\varepsilon > 0$. That is,~$E_\star \subseteq \left\{ x \in \R^{n + 1} : x_{n + 1} = 0 \right\}$, which is clearly impossible. The conclusion of
Theorem~\ref{mainforcatenoids}
is thus valid.
\end{proof}

\subsection{Graph-type~$\alpha$-stationary surfaces}

In this subsection, we prove Theorem~\ref{mainthm} under assumption~\eqref{Egraph}, as specified by the next result.

\begin{theorem} \label{mainforgraphs}
If~$E$ is an~$\alpha$-stationary set in~$\R^{n + 1}$ such that~\eqref{Egraph} holds true for some continuous function~$\varphi: [0, +\infty] \to (0, +\infty)$ satisfying~\eqref{phisublin}, then~$E$ is a half-space.
\end{theorem}

Theorem~\ref{mainforgraphs} is a consequence of a blow-down argument analogous to the one presented for instance in~\cite[Lemma~3.1]{FV17}. The technical
details go as follows.

\begin{proof}[Proof of Theorem~\ref{mainforgraphs}]
Up to a vertical translation, we can suppose that~$0 \in \partial E$. Notice that condition~\eqref{Egraph} continues to hold, for a possibly different function~$\varphi$ still fulfilling~\eqref{phisublin}.

We consider the rescaled set
\begin{equation} \label{ERdef}
E_R := \frac{E}{R} = \Big\{ y \in \R^{n + 1} : R y \in E \Big\},
\end{equation}
for~$R > 0$ large. Notice that~$E_R$ is~$\alpha$-stationary,~$0 \in \partial E_R$, and, by~\eqref{Egraph},
$$
\left\{ y \in \R^{n + 1} : y_{n + 1} < - \frac{\varphi(R |y'|)}{R} \right\} \subseteq E_R \subseteq \left\{ y \in \R^{n + 1} : y_{n + 1} < \frac{\varphi(R |y'|)}{R} \right\}.
$$
Note that, under our assumptions on~$\varphi$, for every~$\varepsilon > 0$ there exists a constant~$C_\varepsilon > 0$ for which we have~$\varphi(t) \le C_\varepsilon + (\varepsilon/2) t$ for all~$t \ge 0$. Consequently, for every~$\varepsilon \in (0, 1/4)$ there exists~$R_\varepsilon > 0$ such that
\begin{equation} \label{epsflat}
\big\{ y \in B_1 : y_{n + 1} < - \varepsilon \big\} \subseteq E_R \cap B_1 \subseteq \big\{ y \in B_1 : y_{n + 1} < \varepsilon \big\}
\end{equation}
holds true for all~$R \ge R_\varepsilon$.

Let now~$\varepsilon$ be lower or equal to the parameter~$\varepsilon_0$ of~\cite[Theorem~6.1]{CRS10}. We remark that such result, stated in~\cite{CRS10} for minimizers of the~$\alpha$-perimeter, is also valid for~$\alpha$-stationary sets satisfying \eqref{epsflat}, as revealed by a careful inspection of its proof---this fact has also been observed in~\cite{CCS17}. Therefore, we can employ~\cite[Theorem~6.1]{CRS10} in our setting and deduce that, for~$R$ sufficiently large,~$\partial E_R$ can be written inside the ball~$B_{1/2}$ as the graph of a function~$u_R: B_{1/2}' \to \R$ satisfying~$u_R(0) = 0$ and~$\| u_R \|_{C^{1, \beta}(B_{1/2}')} \le C$, for some constants~$\beta \in (0, 1)$ and~$C > 0$ depending only on~$n$ and~$\alpha$.

By definition~\eqref{ERdef}, we conclude that~$\partial E$ is the graph of the function~$u: \R^n \to \R$ satisfying~$u(x') = R \, u_R(x' / R)$ for all~$x' \in B_{R/4}'$. By this and the uniform bound on the~$C^{1, \beta}$ norm of~$u_R$, we find that
$$
[\nabla u]_{C^\beta(B'_{R'/4})} = \frac{[\nabla u_R]_{C^\beta(B'_{1/4})}}{R^\beta} \le \frac{C}{R^{\beta}} \quad \mbox{for all large } R.
$$
Hence,~$\nabla u$ is constant in~$\R^n$ and~$E$ is a half-space.
\end{proof}

\vfill

\end{document}